\RequirePackage{fix-cm}
\documentclass[smallextended]{svjour3}       
\smartqed  
\usepackage{graphicx}
\usepackage{amsmath,amssymb}
\usepackage{enumerate}
\usepackage{url}

\RequirePackage[numbers]{natbib}
\RequirePackage[colorlinks,citecolor=blue,linkcolor=blue,urlcolor=blue]{hyperref}
\DeclareMathOperator{\supp}{supp}

\DeclareMathOperator{\Laplas}{\mathrm{\Delta}}

\newcommand{\R}{\mathbb{R}}
\newcommand{\N}{\mathbb{N}}
\newcommand{\I}{\mathbb{I}}
\newcommand{\p}{\mathbb{P}}
\newcommand{\MF}{\mathcal{M}_F}

\newcommand{\Cf}{\mathrm{C}}
\newcommand{\eps}{\varepsilon}
\newcommand{\NM}{\mathcal{N}}
\newcommand{\id}{\mathrm{id}}
\newcommand{\F}{\mathcal{F}}

\begin{document}

\title{On Dean-Kawasaki Dynamics with\\ Smooth Drift Potential
}


\author{Vitalii Konarovskyi \and Tobias Lehmann \and Max von Renesse}


\institute{V. Konarovskyi, T. Lehmann, M. von Renesse \at
              Universit\"{a}t Leipzig, Fakult\"{a}t f\"{u}r Mathematik und Informatik,\\ Augustusplatz 10, 04109 Leipzig, Germany \\
              \email{konarovskyi@gmail.com, tobias.lehmann@math.uni-leipzig.de,\\ renesse@uni-leipzig.de}           
}

\date{\today}

\maketitle

\begin{abstract}
We consider the Dean-Kawasaki equation with smooth drift interaction potential and show that measure valued solutions exist only in certain parameter regimes in which case they are given by  finite Langevin particle systems with mean field interaction. 
\keywords{Dean-Kawasaki equation \and Langevin dynamics \and Wasserstein diffusion \and It\^o formula for measure-valued processes}
\subclass{60H15 \and 60K35 \and 82C22 \and 60G57 \and 82C31}
\end{abstract}

\section{Introduction and main result}

This paper is devoted to the existence, uniqueness and structure of solutions to the Dean-Kawasaki equation 

\begin{equation}\label{f_DK_equation}
d\mu_t=\frac{\alpha}{2}\Laplas\mu_tdt+\nabla\cdot\left(\mu_t\nabla\frac{\delta F(\mu_t)}{\delta\mu_t}\right)dt+\nabla\cdot\left(\sqrt{\mu_t}dW_t\right),
\end{equation}
which appears in macroscopic fluctuation theory or models for glass dynamics in non-equilibrium statistical physics~\cite{MR2095422,
RevModPhys.87.593,
doi:10.1063/1.4913746,
Dean:1996,
MR3744636,
1742-5468-2014-4-P04004,
doi:10.1063/1.4883520,
MR3813113,
MR1661764,
MR3280005,
Kawasaki199435,
PhysRevE.89.012150,
MR978701,
doi:10.1063/1.478705,
0953-8984-12-8A-356,
Rotskoff2018,
1742-5468-2016-11-113202,
Solon2015,
Spohn:1991,
MR2452196}. Here $dW$ denotes a space-time white noise vector field and  $\frac{\delta F(\mu)}{\delta\mu}$ denotes the functional derivative of $F$. \smallskip 

Extending our previous result for the non-interacting case in~\cite{Konarovskyi:DK:2018}, we show that for smooth potentials $F$ measure valued solutions to \eqref{f_DK_equation} exist only for a discrete range of parameters $\alpha$ in which case the solution is given in terms of a finite particle system. 

The precise definition of a (weak martingale) solution to \eqref{f_DK_equation}
and our main result read as follows. 
\begin{definition}
A continuous $\MF(\R^d)$-valued process $\mu_t$ is a solution to equation~\eqref{f_DK_equation}, if for each $\varphi\in\Cf^2_b(\R^d)$ the process 
$$
M_{\varphi}(t):=\langle\varphi,\mu_t\rangle-\int_0^t\left[\frac{\alpha}{2}\langle\Laplas\varphi,\mu_s\rangle+\left\langle\nabla\varphi\cdot\nabla\frac{\delta F(\mu_s)}{\delta\mu_s},\mu_s\right\rangle\right]ds,\quad t\geq 0,
$$
is a continuous martingale with respect to the filtration $\F_t:=\sigma(\mu_s,\ s\in[0,t])$, $t\geq 0$, with the quadratic variation
\begin{equation}\label{f_variation_of_M_varphi}
[M_{\varphi}]_t=\int_0^t\left\langle\left|\nabla\varphi\right|^2,\mu_s\right\rangle ds,\quad t\geq 0.
\end{equation}
\end{definition}

Let $\Cf^{2,2}_b(\MF(\R^d))$ denote the space of twice continuously differentiable functions on $\MF(\R^d)$, which are bounded on the subsets $\{\mu\in\MF(\R^d):\ \mu(\R^d)\leq C\}$, $C>0$, together with their derivatives. For the precise definition of $\Cf^{2,2}_b(\MF(\R^d))$ see Section~\ref{section_preliminaries}.

\begin{theorem}[Existence and uniqueness of solutions to the Dean-Kawasaki equation]\label{thm_existence_and_uniqueness}
Let $\nu\in\MF(\R^d)$, $b:=\nu(\R^d)\not=0$ and $F\in\Cf^{2,2}_b(\MF(\R^d))$. Then the Dean-Kawasaki equation
$$
d\mu_t=\frac{\alpha}{2}\Laplas\mu_tdt+\nabla\cdot\left(\mu_t\nabla\frac{\delta F(\mu_t)}{\delta\mu_t}\right)dt+\nabla\cdot\left(\sqrt{\mu_t}dW_t\right)
$$
has a (unique in law) solution $\mu_t$, $t\geq 0$, starting from $\nu$, i.e. $\mu_0=\nu$, if and only if $b\alpha=:n\in\N$ and $\nu=\frac{b}{n}\sum_{i=1}^n\delta_{x^i}$ for some $x^i\in\R^d$, $i\in[n]:=\{1,\ldots,n\}$. Moreover, 
$$
\mu_t=\frac{b}{n}\sum_{i=1}^n\delta_{X^i(t)},\quad t\geq 0,
$$
where $X(t)=(X^1(t),\ldots,X^n(t))$, $t\geq 0$, is a (unique) solution to the equation
$$
dX^i(t)=-\nabla\frac{\delta F(\mu_t)}{\delta\mu_t}(X^i(t))dt+\sqrt{\frac{n}{b}}dw^i(t),\quad i\in[n],
$$ 
with $X(0)=(x^1,\ldots,x^n)$, and $w^i(t)$, $t\geq 0$, $i\in[n]$, are independent standard Wiener processes on $\R^d$.
\end{theorem}

 We remark that the statement above is false for completely arbitrary drift $F$, since Dean-Kawasaki models with singular drift admitting complex solutions are known e.g.~\cite{MR2537551} and   \cite{MR2606878,Konarovskyi_CM:2017,Konarovskyi_SR:2017,Konarovskyi_LDP:2015,Marx2017,Schiavo2018} both in case of $\alpha>0$ or $\alpha =0$, respectively. We also note that the regularised versions of the Dean-Kawasaki equation can admit non-trivial solutions (see, e.g.~\cite{Zimmer:2018,ZimmerII:2018,Gess:2017}).  \\

\textit{Contents of the paper.} The  proof of our main theorem is based on a reduction to the simpler case when $F=0$, which was  treated in \cite{Konarovskyi:DK:2018},  by means of a Girsanov transform which is combined with an appropriate It\^o formula for $F(\mu)$. The latter is obtained by means of an explicit approximation of smooth functionals $F$ by simple cylindrical functionals in terms of measure valued versions of Bernstein polynomials, which is given in  the appendix and which might be of independent mathematical interest.

%
%
%

%


\section{Preliminaries}\label{section_preliminaries}

Let $\Cf(K)$ be the space of continuous functions on a closed subset $K$ of $\R^d$, Usually, $K$ will be a rectangle $[a,b]^d$ or $\R^d$. The set of bounded continuous functions on $K$ is denoted by $\Cf_b(K)$. If $K$ is a compact set, then trivially $\Cf_b(K)=\Cf(K)$. For $m\in\N$ we define by  $\Cf^m(K)$ the space of $m$ times continuously differentiable functions on the interior of $K$ and which can be extended to continuous functions on $K$. We say that $f$ is smooth on $K$ if it belongs to $C^m(K)$ for all $m\geq 1$. The set of smooth functions on $K$ is denoted by $C^{\infty}(K)$. If 
$l=(l_1,\ldots,l_d)\in(\N\cup\{0\})^d=:\N_0^d$ and $l\not=(0,\ldots,0)$, then we will use the notation
$$
\frac{\partial^{|l|}f}{\partial x^l}=\frac{\partial^{l_1+\ldots+l_d}f}{\partial x_1^{l_1}\ldots\partial x_d^{l_d}}
$$
for the corresponding derivative of $f$ if it exists. We also set $f^{((0,\ldots,0))}=f$ and $\Cf^0(K):=\Cf(K)$. If $K=[a,b]^d$, then we equip $\Cf^m(K)$ with the uniform norm denoted by $\|\cdot\|_{\Cf^m(K)}$. If $K=\R^d$, the topology on $\Cf^m(K)$ is generated by the seminorms of uniform convergence on compact sets.

We will denote the set of finite measures on $K$ by $\MF(K)$ (or shortly $\MF$). For each $\varphi\in \Cf_b(K)$ we set
$$
\langle\varphi,\mu\rangle:=\int_{\R^d}\varphi(x)\mu(dx).
$$
We equip $\MF$ with the weak topology defined by 
$$
\mu_n\to\mu\ \mbox{in}\ \MF,\ n\to\infty,\quad\mbox{iff}\quad\langle\varphi,\mu_n\rangle\to\langle\varphi,\mu\rangle,\ n\to\infty,\ \forall\varphi\in\Cf_b(K).
$$
It is well known that such a  topology is metrisable and $\MF$ is a Polish space. 

\begin{lemma}\label{lemma_NM_C}
If $K$ is a compact set, then for each $C>0$ the set $\NM_C(K):=\{\mu\in\MF(K):\ \mu(K)\leq C\}$ is compact in $\MF(K)$.
\end{lemma}

Let $\Cf(\MF)$ be the set of continuous functions from $\MF(K)$ to $\R$.

If $K$ is compact, we equip the space $\Cf(\MF(K))$ with the topology of uniform convergence on compact sets $\NM_C(K)$, $C>0$. Then one can prove that $\Cf(\MF(K))$ is a Polish space. 

A function $F\in\Cf(\MF(K))$ is said to be {\it differentiable} if for every $\mu\in\MF(K)$ 
$$
F'(\mu;x):=\frac{\delta F(\mu)}{\delta\mu}(x):=\frac{\partial}{\partial\eps}F(\mu+\eps\delta_x)|_{\eps=0}=\lim_{\eps\to 0+}\frac{F(\mu+\eps\delta_x)-F(\mu)}{\eps}.
$$
exists for each $x\in K$ and belongs to $\Cf(K)$. The set of functions for which $F'(\mu;x)$  is jointly continuous in $\mu$ and $x$ is denoted by $\Cf^1(\MF)$. 

Similarly, we can define the second order derivative. So, the {\it second derivative} of a function $F\in\Cf(\MF)$ is defined by
$$
F''(\mu;x,y):=\frac{\delta^2 F(\mu)}{\delta\mu^2}(x,y):=\frac{\partial^2}{\partial\eps_1\partial\eps_2}F(\mu+\eps_1\delta_x+\eps_2\delta_y)|_{\eps_1=\eps_2=0},
$$
if it exists for all $x,y\in K$ and belongs to $\Cf(K^2)$. The set of functions from $\Cf^1(\MF)$ for which $F''(\mu;x,y)$  is jointly continuous in $\mu$, $x$ and $y$ is denoted by $\Cf^2(\MF)$. The notion of differentiable functions on $\MF$ was taken from~\cite[Section~2]{MR1242575}.

We also set for $m\in\N\cup\{\infty\}$
\begin{align*}
\Cf^{1,m}(\MF)=\left\{F\in C^1(\MF):\ \begin{array}{l}
F'(\mu;\cdot)\in\Cf^m(K)\ \ \forall\mu\in\MF\\ 
\mbox{and its derivatives (w.r.t.}\ x\mbox{)}\\ 
\mbox{are jointly continuous in}\  \mu, x
\end{array}
\right\}
\end{align*}
and
$$
\Cf^{2,m}(\MF)=\left\{F\in \Cf^{1,m}(\MF)\cap\Cf^2(\MF):\ 
\begin{array}{l}
F''(\mu;\cdot)\in\Cf^m(K^2)\ \ \forall\mu\in\MF\\ 
\mbox{and its derivatives (w.r.t.}\ x,y\mbox{)}\\ 
\mbox{are jointly continuous in}\  \mu, x,y
\end{array}
\right\}
$$
Let $\Cf^{0,m}(\MF):=\Cf(\MF)$ for each $m\in\N\cup\{\infty\}$. 

We denote by $\Cf^{2,m}_b(\MF)$ the set of functions $F$ from $\Cf^{2,m}(\MF)$ such that for each $C>0$ $F$, $F'$ and $F''$ together with their derivatives up to the order $m$ are bounded on $\NM_C(\MF(K))$, $\NM_C(\MF(K))\times K$ and $\NM_C(\MF(K))\times K^2$, respectively.

\section{It\^{o} formula for the Dean-Kawasaki equation}\label{section_Ito_formula}

Let $\MF:=\MF(\R^d)$ and $F$ be a function from $\Cf^{2,2}_b(\MF)$.

In this section, we are going to establish the It\^{o} formula for a solution to the Dean-Kawasaki equation~\eqref{f_DK_equation}. We recall that a continuous $\MF$-valued process $\mu_t$ is a solution to equation~\eqref{f_DK_equation}, if for each $\varphi\in\Cf^2_b(\R^d)$ the process 
$$
M_{\varphi}(t):=\langle\varphi,\mu_t\rangle-\int_0^t\left[\frac{\alpha}{2}\langle\Laplas\varphi,\mu_s\rangle+\left\langle\nabla\varphi\cdot\nabla\frac{\delta F(\mu_s)}{\delta\mu_s},\mu_s\right\rangle\right]ds,\quad t\geq 0,
$$
is a continuous martingale with respect to the filtration $\F_t:=\sigma(\mu_s,\ s\in[0,t])$, $t\geq 0$, with the quadratic variation
\begin{equation}\label{f_variation_of_M_varphi}
[M_{\varphi}]_t=\int_0^t\left\langle\left|\nabla\varphi\right|^2,\mu_s\right\rangle ds,\quad t\geq 0.
\end{equation}

\begin{remark}\label{rem_conservation}
Taking $\varphi(x)=1$, $x\in\R^d$, it is easy to see that $\langle\varphi,\mu_t\rangle=\langle\varphi,\mu_0\rangle$ for all $t\geq 0$. In particular, $\mu_t\in\NM_C(\R^d)$ for all $t\geq 0$ if $\mu_0\in\NM_C(\R^d)$.
\end{remark}

\begin{theorem}[It\^{o} formula for the Dean-Kawasaki equation]\label{thm_Ito_formula}
For every $G\in\Cf^{2,2}_b(\MF)$ the following process
\begin{equation}\label{f_M_G}
\begin{split}
M^G(t):=G(\mu_t)-G(\mu_0)&-\int_0^t\left[\frac{\alpha}{2}\left\langle\Laplas\frac{\delta G(\mu_s)}{\delta\mu_s},\mu_s\right\rangle+\left\langle\nabla\frac{\delta G(\mu_s)}{\delta\mu_s}\cdot\nabla\frac{\delta F(\mu_s)}{\delta\mu_s},\mu_s\right\rangle\right.\\
&\left.+\frac{1}{2}\left\langle\int_{\R^d}\nabla_x\cdot\nabla_y\frac{\delta^2 G(\mu_s)}{\delta\mu_s^2}\delta_x(dy),\mu_s\right\rangle\right]ds,\quad t\geq 0,
\end{split}
\end{equation}
is a continuous $(\F_t)$-martingale with the quadratic variation
$$
[M^G]_t=\int_0^t\left\langle\left|\nabla\frac{\delta G(\mu_s)}{\delta\mu_s}\right|^2,\mu_s\right\rangle ds,\quad t\geq 0.
$$
\end{theorem}

\begin{proof}
We first prove the theorem for a function $G$ of the form
\begin{equation}\label{f_partial_case_of_G}
G(\mu)=f\left(\langle\varphi_1,\mu\rangle,\ldots,\langle\varphi_n,\mu\rangle\right)=f\left(\langle\varphi,\mu\rangle\right),\quad \mu\in\MF,
\end{equation}
where $f\in\Cf^2_b(\R^d)$, \ $\varphi_i$, $i\in\{1,\ldots,n\}$, are smooth functions on $\R^d$ with compact supports and $\langle\varphi,\mu\rangle:=(\langle\varphi_1,\mu\rangle,\ldots,\langle\varphi_n,\mu\rangle)$.  By the It\^{o} formula for real valued semimartingales, we have
\begin{align*}
dG(\mu_t)&=df\left(\langle\varphi_1,\mu_t\rangle,\ldots,\langle\varphi_n,\mu_t\rangle\right)=\frac{\alpha}{2}\sum_{i=1}^n\frac{\partial f}{\partial z_i}(\langle\varphi,\mu_t\rangle)\langle\Laplas\varphi_i,\mu_t\rangle dt\\
&+\sum_{i=1}^n\frac{\partial f}{\partial z_i}(\langle\varphi,\mu_t\rangle)\left\langle\nabla\varphi_i\cdot\nabla\frac{\delta F(\mu_t)}{\delta\mu_t},\mu_t\right\rangle dt\\
&+\frac{1}{2}\sum_{i,j=1}^n\frac{\partial^2 f}{\partial z_i\partial z_j}(\langle\varphi,\mu_t\rangle)\left\langle\nabla\varphi_i\cdot\nabla\varphi_j,\mu_t\right\rangle dt+\sum_{i=1}^n\frac{\partial f}{\partial z_i}(\langle\varphi,\mu_t\rangle)dM_{\varphi_i}(t).
\end{align*}
Next, using the equalities
$$
\frac{\delta G(\mu)}{\delta\mu}(x)=\sum_{i=1}^n\frac{\partial f}{\partial z_i}(\langle\varphi,\mu_t\rangle)\varphi_i(x),\quad \mu\in\MF,\ \ x\in\R^d,
$$
and
$$
\frac{\delta^2 G(\mu)}{\delta\mu^2}(x,y)=\sum_{i,j=1}^n\frac{\partial^2 f}{\partial z_i\partial z_j}(\langle\varphi,\mu_t\rangle)\varphi_i(x)\varphi_j(y),\quad \mu\in\MF,\ \ x,y\in\R^d,
$$
it is easy to see that
\begin{align*}
dG(\mu_t)&=\frac{\alpha}{2}\left\langle\Laplas\frac{\delta G(\mu_t)}{\delta\mu_t},\mu_t\right\rangle dt+\left\langle\nabla\frac{\delta G(\mu_t)}{\delta\mu_t}\cdot\nabla\frac{\delta F(\mu_t)}{\delta\mu_t},\mu_t\right\rangle dt\\
&+\frac{1}{2}\left\langle\int_{\R^d}\nabla_x\cdot\nabla_y\frac{\delta^2 G(\mu_t)}{\delta\mu_t^2}\delta_x(dy),\mu_t\right\rangle dt+\sum_{i=1}^n\frac{\partial f}{\partial z_i}(\langle\varphi,\mu_t\rangle)dM_{\varphi_i}(t).
\end{align*}
Moreover, the quadratic variation of $M^G(t)$, $t\geq 0$, the martingale part of $G(\mu_t)$, $t\geq 0$, is equal to
\begin{align*}
[M^G]_t&=\int_0^t\sum_{i,j=1}^n\frac{\partial f}{\partial z_i}(\langle\varphi,\mu_s\rangle)\frac{\partial f}{\partial z_j}(\langle\varphi,\mu_s\rangle)\left\langle\nabla\varphi_i\cdot\nabla\varphi_j,\mu_s\right\rangle ds\\
&=\int_0^t\left\langle\left|\nabla\frac{\delta G(\mu_s)}{\delta\mu_s}\right|^2,\mu_s\right\rangle ds.
\end{align*}
Thus, the It\^{o} formula holds for any function $G$ given by~\eqref{f_partial_case_of_G}.

Next, by Theorem~\ref{thm_approximation} and Remark~\ref{remark_about_extention}, there exists a sequence $\{G_n\}_{n\geq 1}$ of functions of the form~\eqref{f_partial_case_of_G} such that for all $\mu\in\MF$
$$
G_n(\mu)\to G(\mu),\quad n\to\infty,
$$
$$
\frac{\delta G_n(\mu)}{\delta\mu}\to\frac{\delta G(\mu)}{\delta\mu}\ \ \mbox{in}\ \ C^2(\R^d),\quad n\to\infty,
$$
and
$$
\frac{\delta^2 G_n(\mu)}{\delta\mu^2}\to\frac{\delta^2 G(\mu)}{\delta\mu^2}\ \ \mbox{in}\ \ C^2(\R^{2d}),\quad n\to\infty, 
$$
Moreover, $G_n$, $\frac{\delta G_n(\mu)}{\delta\mu}$ and 
$\frac{\delta^2 G_n(\mu)}{\delta\mu^2}$ and their derivatives (by $x$ and $y$) are uniformly bounded (in $n$) on $\NM_C(\MF(\R^d))$, $\NM_C(\MF(\R^d))\times\R^d$ and  ${\NM_C(\MF(\R^d))\times\R^{2d}}$, respectively. 
This implies that for each $\mu\in\MF$
$$
\left\langle\Laplas\frac{\delta G_n(\mu)}{\delta\mu},\mu\right\rangle\to\left\langle\Laplas\frac{\delta G(\mu)}{\delta\mu},\mu\right\rangle,\quad n\to\infty,
$$
$$
\left\langle\nabla\frac{\delta G_n(\mu)}{\delta\mu}\cdot\nabla\frac{\delta F(\mu)}{\delta\mu},\mu\right\rangle\to\left\langle\nabla\frac{\delta G(\mu)}{\delta\mu}\cdot\nabla\frac{\delta F(\mu)}{\delta\mu},\mu\right\rangle,\quad n\to\infty,
$$
$$
\left\langle\int_{\R^d}\nabla_x\cdot\nabla_y\frac{\delta^2 G_n(\mu)}{\delta\mu^2}\delta_x(dy),\mu\right\rangle\to\left\langle\int_{\R^d}\nabla_x\cdot\nabla_y\frac{\delta^2 G(\mu)}{\delta\mu^2}\delta_x(dy),\mu\right\rangle,\quad n\to\infty,
$$
and
$$
\left\langle\left|\nabla\frac{\delta G_n(\mu)}{\delta\mu}\right|^2,\mu\right\rangle\to \left\langle\left|\nabla\frac{\delta G(\mu)}{\delta\mu}\right|^2,\mu\right\rangle,\quad n\to\infty,
$$
by the dominated convergence theorem.

Using the uniform boundedness of $G_n$ and its derivatives, Remark~\ref{rem_conservation} and the dominated convergence theorem, we obtain that the It\^{o} formula for $G$ is also valid. The theorem is proved.
\end{proof}

\section{Girsanov's transformation and proof of the main result}\label{section_Girsanov_transformation}

We assume that a solution $\mu_t$, $t\geq 0$, to equation~\eqref{f_DK_equation} is a canonical process on the filtered probability space $(\Omega,\F,(\F_t)_{t\geq 0},\p)$, where $\Omega$ is the space of continuous functions from $[0,+\infty)$ to $\MF:=\MF(\R^d)$, $\p$ is the distribution of the process $\mu_t$, $t\geq 0$, $(\F_t)_{t\geq 0}$ is the right-continuous and complete induced filtration generated by $\mu_t$, $t\geq 0$, and $\F=\vee_{t\geq 0}\F_t$. We remark that such a filtration exists by Lemma~7.8~\cite{Kallenberg:2002}.

Now, let $N(t)$, $t\geq 0$, be a continuous nonnegative martingale with ${N(0)=1}$. We consider a new measure on $(\Omega,\F)$ defined as
$$
d\p_N:= N(t)d\p\quad \mbox{on}\ \ \F_t,\ \ t\geq 0,
$$
that is, $\p_N(A)=\int_{A}N(t)d\p$, $A\in\F_t$,
which exists by Lemma~18.18~\cite{Kallenberg:2002}. Next we take any function $G$ from $\Cf^{2,2}_b(\MF)$ and note that 
$$
E^G(t):=e^{M^G(t)-\frac{1}{2}[M^G]_t},\quad t\geq 0,
$$
is a continuous $(\F_t)$-matringale with $E^G(0)=1$, by Novikov's theorem (see Theorem~18.23~\cite{Kallenberg:2002}). Here $M^G$ is given by~\eqref{f_M_G}. So, we can define the measure $\p^G:=\p_{E^G}$ on $(\Omega,\F)$, that is, 
\begin{equation}\label{f_measure_P_G}
d\p^G:= E^G(t)d\p=e^{M^G(t)-\frac{1}{2}[M^G]_t}d\p\quad \mbox{on}\ \ \F_t,\ \ t\geq 0.
\end{equation}

\begin{theorem}[Girsanov's transformation for solutions to the Dean-Kawasaki equation]\label{theorem_Girsanov's_transformation}
Let $G$ be a function from $\Cf^{2,2}_b(\MF)$ and $\p^G$ be defined by~\eqref{f_measure_P_G}. Then the process $\mu_t$, $t\geq 0$, solves the equation
$$
d\mu_t=\frac{\alpha}{2}\Laplas\mu_tdt+\nabla\cdot\left(\mu_t\nabla\frac{\delta (F+G)(\mu_t)}{\delta\mu_t}\right)dt+\nabla\cdot\left(\sqrt{\mu_t}dW_t\right)
$$
on the probability space $(\Omega,\F,\p^G)$. In particular, $\mu_t$, $t\geq 0$, is a solution to the equation
$$
d\mu_t=\frac{\alpha}{2}\Laplas\mu_tdt+\nabla\cdot\left(\sqrt{\mu_t}dW_t\right)
$$
on $(\Omega,\F,\p^G)$, if $G=-F$.
\end{theorem}

\begin{proof}
To prove the statement, we use Girsanov's transformation (see e.g. Theorem~18.19 and Lemma~18.21~\cite{Kallenberg:2002}) and Theorem~\ref{thm_Ito_formula}. So, we take a function $\varphi\in\Cf_b(\R^b)$ and compute the joint quadratic variation $[M_{\varphi},M^G]_t$ using Theorem~\ref{thm_Ito_formula}. The polarisation equality implies
$$
[M_{\varphi},M^G]_t=\int_0^t\left\langle\nabla\varphi\cdot\nabla\frac{\delta G(\mu_s)}{\delta\mu_s},\mu_s\right\rangle ds,\quad t\geq 0.
$$
Thus, by Theorem~18.19 and Lemma~18.21~\cite{Kallenberg:2002}, the process
$$
M_{\varphi}(t)-[M_{\varphi},M^G]_t=\langle\varphi,\mu_t\rangle-\int_0^t\left[\frac{\alpha}{2}\langle\Laplas\varphi,\mu_s\rangle+\left\langle\nabla\varphi\cdot\nabla\frac{\delta (F+G)(\mu_s)}{\delta\mu_s},\mu_s\right\rangle\right]ds
$$
is a continuous $(\F_t)$-martingale on $(\Omega,\F,\p^G)$ with the quadratic variation given by~\eqref{f_variation_of_M_varphi}.
\end{proof}

\begin{proof}[Proof of Theorem~\ref{thm_existence_and_uniqueness}]
We assume that $\mu_t$, $t\geq 0$, is a solution to the Dean-Kawasaki equation. Applying Theorem~\ref{theorem_Girsanov's_transformation} for $G=-F$, we obtain that $\mu_t$, $t\geq 0$, must solve the equation
\begin{equation}\label{f_DK_eq_for_F=0}
d\mu_t=\frac{\alpha}{2}\Laplas\mu_tdt+\nabla\cdot\left(\sqrt{\mu_t}dW_t\right)
\end{equation}
on the space $(\Omega,\F,\p^{-F})$. By a simple computation, it is easy to see that the process $\tilde{\mu}_t:=\frac{1}{b}\mu_{bt}$, $t\geq 0$, is a solution to~\eqref{f_DK_eq_for_F=0} with the parameter $b\alpha$ instead of $\alpha$. Moreover, $\tilde{\mu}_t$, $t\geq 0$, takes values in the space of probability measures on $\R^d$. Hence, by Theorem~1~\cite{Konarovskyi:DK:2018}, $\beta\alpha=n\in\N$ and there exists a family of $\R^d$-valued processes $\tilde{X}^i(t)$, $t\geq 0$, such that
$$
\tilde{\mu}_t=\frac{1}{n}\sum_{i=1}^n\delta_{\tilde{X}^i(t)},
$$ 
with $\tilde{X}^i(t)=x^i+\tilde{w}^i(nt)$, $t\geq 0$, $i\in[n]$, and $\tilde{w}^i(t)$, $t\geq 0$, $i\in[n]$, are standard independent $(\F_t)$-Wiener processes on $\R^d$.
This implies that 
$$
\mu_t=\frac{b}{n}\sum_{i=1}^n\delta_{\tilde{X}^i\left(\frac{t}{b}\right)}=\frac{b}{n}\sum_{i=1}^n\delta_{X^i(t)},
$$
where $X^i(t)=\tilde{X}^i\left(\frac{t}{b}\right)=x^i+\tilde{w}^i\left(\frac{nt}{b}\right)$, $t\geq 0$.

Next, we note that the process $N(t):=-M^{G}(t)+[M^{G}]_t$, $t\geq 0$, is a continuous $(\F_t)$-martingale on $(\Omega,\F,\p^G)$, by Girsanov's transformation. Thus, we can consider the following transformation of measure $\p^G$ given by
$$
d\tilde{\p}:=e^{N(t)-\frac{1}{2}[N]_t}d\p^G=d\p\quad \mbox{on}\ \ \F_t,\ \ t\geq 0.
$$  
Thus, applying Girsanov's theorem to $X^i(t)$, $t\geq 0$, $i\in[n]$, on $(\Omega,\F,\p^G)$, we obtain that
$$
R^i(t):=X^i(t)+[M^G,X^i]_t=X^i(t)+\int_0^t\nabla\frac{\delta F(\mu_s)}{\delta\mu_s}(X^i(s))ds,\quad t\geq 0,
$$
are $\R^d$-valued continuous $(\F_t)$-martingales on $(\Omega,\F,\tilde{\p}=\p)$ for all $i\in[n]$ and $[R^i,R^j]_t=\frac{n}{b}I\I_{\{i=j\}}$, $t\geq 0$, $i,j\in[n]$, where $I$ denotes the identity $d\times d$ matrix.

The uniqueness also trivially follows from Girsanov's transformation.
\end{proof}

\begin{example}
We assume that $V_1,V_2\in\Cf_b^2(\R^d)$, $V_1(x)=V_1(-x)$, $x\in\R^d$, and take
$$
F(\mu):=\frac{1}{2}\int_{\R^d}\int_{\R^d}V_1(x-y)\mu(dx)\mu(dy)+\int_{\R^d}V_2(x)\mu(dx),\quad\mu\in\MF(\R^d).
$$
In this case,
$$
\frac{\delta F(\mu)}{\delta\mu}(x)=\int_{\R^d}V_1(x-y)\mu(dy)+V_2(x),\quad\mu\in\MF(\R^d),\ \ x\in\R^d,
$$
and
$$
\frac{\delta^2 F(\mu)}{\delta\mu^2}(x,y)=V_1(x-y),\quad\mu\in\MF(\R^d),\ \ x,y\in\R^d.
$$
Then the Dean-Kawasaki equation for interacting Brownian particles has a form
$$
d\mu_t=\frac{\alpha}{2}\Laplas\mu_tdt+\nabla\cdot\left(\mu_t\int_{\R^d}\nabla V_1(\cdot-y)\mu_t(dy)\right)dt+\nabla\cdot\left(\mu_t\nabla V_2\right)dt+\nabla\cdot\left(\sqrt{\mu_t}dW_t\right),
$$
where $V_1$ plays a role of a two-body interaction potential between particles and $V_2$ is an external potential (see e.g.~\cite{MR2095422,
Dean:1996,
MR3744636,
doi:10.1063/1.4883520,
0305-4470-33-15-101,
PhysRevE.91.022130,
doi:10.1063/1.478705}). 

Since $F\in\Cf^{2,2}_b(\MF(\R^d))$, the Dean-Kawasaki equation has a (unique in law) solution if and only if $b\alpha=n\in\N$ and $\mu_0=\frac{b}{n}\sum_{i=1}^n\delta_{x^i}$ for some $x^i\in\R^d$, $i\in[n]$, where $b=\mu_0(\R^d)$, by Theorem~\ref{thm_existence_and_uniqueness}. Moreover,
$$
\mu_t=\frac{b}{n}\sum_{i=1}^n\delta_{X^i(t)},\quad t\geq 0,
$$
where the family $X^i(t)$, $t\geq 0$, $i\in[n]$, solves the equation
$$
dX^i(t)=-\frac{b}{n}\sum_{j=1}^n\nabla V_1(X^i(t)-X^j(t))dt+\nabla V_2(X^i(t))dt+\sqrt{\frac{n}{b}}dw^i(t),\quad i\in[n].
$$
\end{example}



\appendix

\section{Approximation of differentiable functions on $\MF(\R^d)$}

\subsection{Approximation of differentiable functions on $\MF([a,b]^d)$}

In this section, we fix $a,b\in\R$, $a<b$, and denote $K:=[a,b]^d$ and $\MF:=\MF([a,b]^d)$, for convenience of notation.  We remark that each function from $\Cf(\MF)$ is bounded on $\NM_C:=\NM_C(K)=\{\mu\in\MF:\ \mu(K)\leq C\}$ for all $C>0$, since $\NM_C$ is compact in $\MF$. 

We are going to introduce an analog of the Weierstrass approximation of functions from $\Cf^{k,m}(\MF)$. For this we use multiplicative Bernstein polynomials on $K=[a,b]^d$ (see e.g.~\cite{Veretennikov:2016}). Let $g\in\Cf(K)$. We set for $n\geq 1$
$$
B_n(g)(x)=\sum_{j_1,\ldots j_d=0}^ng\left(a_{j_1,\ldots,j_d}^n\right)\varphi_{j_1,\ldots,j_d}^n(x),\quad x\in K,
$$
where 
$$
a_{j_1,\ldots,j_d}^n=\left(a+\frac{j_1(b-a)}{n},\ldots,a+\frac{j_d(b-a)}{n}\right)
$$
and
$$
\varphi_{j_1,\ldots,j_d}^n(x)=\frac{1}{(b-a)^d}\prod_{k=1}^dC_n^{j_k}(x_k-a)^{j_k}(b-x_k)^{n-j_k},\quad x\in K.
$$
Here $C_n^i=\frac{n!}{i!(n-i)!}$. 

We will consider $B_n$, $n\geq 1$, as linear operators from $\Cf^m(K)$ to $\Cf^m(K)$.

\begin{proposition}\label{proposition_bernstein polynomials}
Let $m\in\N_0$. Then the family of linear operators $B_n:\Cf^m(K)\to\Cf^m(K)$, $n\in\N$, satisfies the following properties:
\begin{enumerate}
\item[(B1)] $\{B_n\}_{n\geq 1}$ is a family of uniformly bounded linear operators on $\Cf^m(K)$;

\item[(B2)] For each $g\in\Cf^m(K)$ and $l\in\N_0^d$, $|l|\leq m$, 
$$
\frac{\partial^{|l|}}{\partial x^l}B_n(g)\to \frac{\partial^{|l|}}{\partial x^l}g\quad\mbox{in}\ \ \Cf(K),\ \ \mbox{as}\ \ n\to\infty,
$$
that is, $B_n(g)\to g$ in $\Cf^m(K)$, as $n\to\infty$.

\item[(B3)] If $g_k\to g$ in $\Cf^m(K)$, $k\to\infty$, then for each $l\in\N_0^d$, $|l|\leq m$,
$$
\frac{\partial^{|l|}}{\partial x^l}B_n(g_k)\to \frac{\partial^{|l|}}{\partial x^l}g\quad\mbox{in}\ \ \Cf(K),\ \ \mbox{as}\ \ n,k\to\infty,
$$
that is, $B_n(g_k)\to g$ in $\Cf^m(K)$, as $n,k\to\infty$.
\end{enumerate}
\end{proposition} 

\begin{proof}
For $K=[0,1]^d$ Property (B2) was proved in~\cite{Veretennikov:2016}. The general case can be obtained by the rescaling. Next, for each $g\in\Cf^m(K)$ Property (B2) implies the boundedness of $\{\|B_n(g)\|_{\Cf^m(K)}\}_{n\geq 1}$. By the Banach-Steinhaus theorem, we obtain (B1). Property (B3) easily follows from (B1) and (B2).
\end{proof}

Now, we introduce an analog of Bernstein polynomials on $\MF$. We set for each $\mu\in\MF$
$$
\chi_n(\mu):=\sum_{j_1,\ldots,j_d=0}^n\left\langle\varphi_{j_1,\ldots,j_d}^n,\mu\right\rangle\delta_{a_{j_1,\ldots,j_d}^n},\quad n\geq 1,
$$
where $\delta_{c}$ is the point measure at $c\in\R^d$, i.e. $\delta_{c}(A)$ equals 1 if $c\in A$ and 0 otherwise. We also define for every $F\in\Cf(\MF)$
\begin{equation}\label{f_P_n}
P_n(F)(\mu):=F\left(\chi_n(\mu)\right),\quad \mu\in\MF,\ \ n\geq 1.
\end{equation}

Setting 
$$
u_n^F(z):=F\left(\sum_{j_1,\ldots,j_d=0}^nz_{j_1,\ldots,j_d}\delta_{a_{j_1,\ldots,j_d}^n}\right),\quad z\in[0,\infty)^{(n+1)^d},
$$
it is easy to see that $u_n^F\in\Cf\left([0,\infty)^{(n+1)^d}\right)$ and
\begin{equation}\label{f_P_n_U}
P_n(F)(\mu)=u_n^F\left(\left(\left\langle\varphi_{j_1,\ldots,j_d}^n,\mu\right\rangle\right)_{j_1,\ldots,j_d=0}^n\right),\quad \mu\in\MF,\ \ n\geq 1.
\end{equation}

We will denote by $\id$ the identity map on $\MF$, that is, $\id(\mu)=\mu$, $\mu\in\MF$.
\begin{proposition}\label{prop_continuity_of_chi}
For each $n\geq 1$ the map $\chi_n:\MF\to\MF$ is continuous and for each sequence $\{\mu_k\}_{k\geq 1}$ converging $\mu$ in $\MF$ one has $\chi_n(\mu_k)\to\mu$ in $\MF$ as $n,k\to\infty$. Moreover, $\chi_n$ maps $\NM_C$ to $\NM_C$ for all $C>0$ and $n\geq 1$. 
\end{proposition}

\begin{remark}
Since the set $\NM_C=\{\mu\in\MF:\ \mu(K)\leq C\}$ is compact in $\MF$, we have that for each $C>0$ $\chi_n\to\id$ uniformly on $\NM_C$ as $n\to\infty$, by Proposition~\ref{prop_continuity_of_chi}.
\end{remark}

\begin{remark}
Proposition~\ref{prop_continuity_of_chi} implies that $P_n$ is a linear map from $\Cf(\MF)$ to $\Cf(\MF)$. 
\end{remark}

\begin{proof}[Proof of Proposition~\ref{prop_continuity_of_chi}]
The continuity of $\chi_n$ is trivial. We take an arbitrary sequence $\{\mu_k\}_{k\geq 1}$ in $\MF$ which converges to $\mu$ and $g\in\Cf(K)$.
Then by Proposition~\ref{proposition_bernstein polynomials},
\begin{align*}
\langle g,\chi_n(\mu_k)\rangle&=\sum_{j_1,\ldots,j_d=0}^ng\left(a_{j_1,\ldots,j_d}^n\right)\left\langle\varphi_{j_1,\ldots,j_d}^n,\mu_k\right\rangle\\
&=\left\langle\sum_{j_1,\ldots,j_d=0}^ng\left(a_{j_1,\ldots,j_d}^n\right)\varphi_{j_1,\ldots,j_d}^n,\mu_k\right\rangle\\
&=\left\langle B_n(g),\mu_k\right\rangle\to\langle g,\mu\rangle,\quad n,k\to\infty,
\end{align*}
since the map $\Cf(K)\times\MF\ni(g,\mu)\mapsto\langle g,\mu\rangle\in\R$ is continuous.

Due to the equality
$$
\chi_n(\mu)(K)=\sum_{j_1,\ldots,j_d=0}^n\left\langle\varphi_{j_1,\ldots,j_d}^n,\mu\right\rangle=\left\langle\sum_{j_1,\ldots,j_d=0}^n\varphi_{j_1,\ldots,j_d}^n,\mu\right\rangle=\left\langle 1,\mu\right\rangle=\mu(K),
$$
$\chi_n$ maps $\NM_C$ to $\NM_C$.
\end{proof}

\begin{proposition}\label{prop_convergence_in_C_M}
For each $F\in\Cf(\MF)$ and $C>0$ we have that $P_n(F)\to F$ uniformly on $\NM_C$ as $n\to\infty$, that is,
$$
\sup_{\mu\in\NM_C}|P_n(F)(\mu)-F(\mu)|\to 0,\quad n\to\infty.
$$
\end{proposition}

\begin{remark}
Proposition~\ref{prop_convergence_in_C_M} yields that for each $F\in\Cf(\MF)$ $P_n(F)\to F$ in $\Cf(\MF)$ as $n\to\infty$.
\end{remark}

\begin{proof}[Proof of Proposition~\ref{prop_convergence_in_C_M}]
We assume that the statement is not true. Then there exist $\eps>0$ and a sequence $\{\mu_n\}_{n\geq 1}$ in $\NM_C$ such that $|P_n(F)(\mu_n)-F(\mu_n)|\geq\eps$ for all $n\geq 1$. Since $\NM_C$ is compact, we may assume that $\mu_n\to\mu$ without loss of generality. But by Proposition~\ref{prop_continuity_of_chi} and the continuity of $F$, we have
$$
P_n(F)(\mu_n)-F(\mu_n)=F(\chi_n(\mu_n))-F(\mu_n)\to F(\mu)-F(\mu)=0,\quad n\to\infty,
$$
which contradicts the assumption.
\end{proof}

We note that the space $\Cf(\NM_C)$ of continuous functions from $\NM_C$ to $\R$ furnished with the uniform norm is a Banach space. It is easy to see that for each $n\geq 1$ the map $P_n$ is a continuous linear operator from $\Cf(\NM_C)$ to $\Cf(\NM_C)$. Indeed, the map $\chi_n$ maps $\NM_C$ to $\NM_C$, by Proposition~\ref{prop_continuity_of_chi}. The continuity trivially follows from the form of $P_n$ (see~\eqref{f_P_n}).

\begin{corollary}
The family $\{P_n\}_{n\geq 1}$ of linear operators on $\Cf(\NM_C)$ is uniformly bounded.
\end{corollary}

\begin{proof}
The corollary immediately follows from Proposition~\ref{prop_convergence_in_C_M} and the Banach-Steinhaus theorem.
\end{proof}

\begin{lemma}\label{lemma_properties_of_U}
Let $F\in\Cf^k(\MF)$ for some $k\in\{1,2\}$. Then the function $u_n^F$ belongs to $\Cf^k\left([0,\infty)^{(n+1)^d}\right)$. Moreover,
$$
\frac{\partial}{\partial z_{i_1,\ldots,i_d}}u_n^F(z)=F'\left(\sum_{j_1,\ldots,j_d=0}^nz_{j_1,\ldots,j_d}\delta_{a_{j_1,\ldots,j_d}^n};a_{i_1,\ldots,i_d}^n\right),\quad z\in[0,\infty)^{(n+1)^d},
$$
for all $i_1,\ldots,i_d$ and 
$$
\frac{\partial^2}{\partial z_{j_1,\ldots,j_d}\partial z_{i_1,\ldots,i_d}}u_n^F(z)=F''\left(\sum_{l_1,\ldots,l_d=0}^nz_{l_1,\ldots,l_d}\delta_{a_{l_1,\ldots,l_d}^n};a_{j_1,\ldots,j_d}^n,a_{i_1,\ldots,i_d}^n\right),\quad z\in[0,\infty)^{(n+1)^d},
$$
for all $j_1,\ldots,j_d$, $i_1,\ldots,i_d$, if $k=2$.
\end{lemma}

\begin{proof}
The proof easily follows from the definition of $F'$ and $F''$.
\end{proof}

\begin{proposition}
Let $F\in\Cf^{k,m}(\MF)$ for some $k\in\{1,2\}$ and $m\geq 0$. Then for every $n\geq 1$ $P_n(F)\in\Cf^{k,\infty}$ and for each $\mu\in\MF$, $x,y\in K$
\begin{align*}
P_n'(F)(\mu;x)&=\sum_{j_1,\ldots,j_d=0}^nF'\left(\chi_n(\mu);a_{j_1,\ldots,j_d}^n\right)\varphi_{j_1,\ldots,j_d}^n(x)\\
&=\sum_{j_1,\ldots,j_d=0}^nP_n\left(F'\left(\cdot;a_{j_1,\ldots,j_d}^n\right)\right)(\mu)\varphi_{j_1,\ldots,j_d}^n(x).
\end{align*}
and if $k=2$
\begin{align*}
P_n''(F)(\mu;x,y)&=\sum_{j_1,\ldots,j_d=0}^n\sum_{i_1,\ldots,i_d=0}^nF''\left(\chi_n(\mu);a_{j_1,\ldots,j_d}^n,a_{i_1,\ldots,i_d}^n\right)\varphi_{j_1,\ldots,j_d}^n(x)\varphi_{i_1,\ldots,i_d}^n(y)\\
&=\sum_{j_1,\ldots,j_d=0}^n\sum_{i_1,\ldots,i_d=0}^nP_n\left(F''\left(\cdot;a_{j_1,\ldots,j_d}^n,a_{i_1,\ldots,i_d}^n\right)\right)(\mu)\varphi_{j_1,\ldots,j_d}^n(x)\varphi_{i_1,\ldots,i_d}^n(y).
\end{align*}
\end{proposition}

\begin{proof}
The proposition follows from the definition of the derivatives $F'$, $F''$, equality~\eqref{f_P_n_U} and Lemma~\ref{lemma_properties_of_U}. Indeed,
\begin{align*}
P_n'(F)(\mu;x)&=\sum_{j_1,\ldots,j_d=0}^n\frac{\partial}{\partial z_{j_1,\ldots,j_d}}u_n^F\left(\left(\left\langle\varphi_{j_1,\ldots,j_d}^n,\mu\right\rangle\right)_{j_1,\ldots,j_d=0}^n\right)\varphi_{j_1,\ldots,j_d}^n(x)\\
&=\sum_{j_1,\ldots,j_d=0}^nP_n\left(F'\left(\cdot;a_{j_1,\ldots,j_d}^n\right)\right)(\mu)\varphi_{j_1,\ldots,j_d}^n(x).
\end{align*}
Similarly, one can obtain the equality for $P_n''(F)(\mu;x,y)$.
\end{proof}

\begin{theorem}\label{theorem_approximation_by_bernstein_polynomials}
Let $F\in\Cf^{k,m}(\MF)$ for some $k\in\{1,2\}$ and $m\geq 0$. Then for each $l,\tilde{l}\in\N_0^d$, $|l|+|\tilde{l}|\leq m$ and $C>0$ one has
\begin{equation}\label{f_conv_of_P_n'}
\sup_{x\in K,\ \mu\in\NM_C}\left|\frac{\partial^{|l|}}{\partial x^l} P_n'(F)(\mu;x)-\frac{\partial^{|l|}}{\partial x^l}F'(\mu;x)\right|\to 0,\quad n\to\infty.
\end{equation}
and if $k=2$
\begin{equation}\label{f_conv_of_P_n''}
\sup_{x,y\in K,\ \mu\in\NM_C}\left|\frac{\partial^{|l|}}{\partial x^l}\frac{\partial^{|\tilde{l}|}}{\partial y^{\tilde{l}}} P_n''(F)(\mu;x,y)-\frac{\partial^{|l|}}{\partial x^l}\frac{\partial^{|\tilde{l}|}}{\partial y^{\tilde{l}}}F''(\mu;x,y)\right|\to 0,\quad n\to\infty.
\end{equation}
\end{theorem}

\begin{proof}
We will prove the theorem similarly as Proposition~\ref{prop_convergence_in_C_M}. We start with~\eqref{f_conv_of_P_n'}. If~\eqref{f_conv_of_P_n'} does not hold, then there exist $\eps>0$ and sequences $\{\mu_n\}_{n\geq 1}\subset\NM_C$, $\{x_n\}_{n\geq 1}\in K$ such that 
\begin{equation}\label{f_P-F}
\left|\frac{\partial^{|l|}}{\partial x^l} P_n'(F)(\mu_n;x_n)-\frac{\partial^{|l|}}{\partial x^l}F'(\mu_n;x_n)\right|\geq\eps
\end{equation}
for all $n\geq 1$. Since $\NM_C$ and $K$ are compact sets, we may assume that $\mu_n\to\mu_0$ and $x_n\to x_0$ as $n\to\infty$, without loss of generality. So, we compute
\begin{align*}
\frac{\partial^{|l|}}{\partial x^l} P_n'(F)(\mu_n;x_n)&=\frac{\partial^{|l|}}{\partial x^l}\sum_{j_1,\ldots,j_d=0}^nF'\left(\chi_n(\mu_n);a_{j_1,\ldots,j_d}^n\right)\varphi_{j_1,\ldots,j_d}^n(x_n)\\
&=\frac{\partial^{|l|}}{\partial x^l}B_n\left(F'\left(\chi_n(\mu_n);\cdot\right)\right)(x_n)
\end{align*}
Since $F'$ is continuous on $\MF\times K$ and $K$ is compact, it is easy to see that $F'(\chi_n(\mu_n);\cdot)\to F'(\mu_0;\cdot)$ in $\Cf(K)$ as $n\to\infty$, using Proposition~\ref{prop_continuity_of_chi}. Thus, by Proposition~\ref{proposition_bernstein polynomials}~(B3),
$$
\frac{\partial^{|l|}}{\partial x^l} P_n'(F)(\mu_n;x_n)=\frac{\partial^{|l|}}{\partial x^l}B_n\left(F'\left(\chi_n(\mu_n);\cdot\right)\right)(x_n)\to \frac{\partial^{|l|}}{\partial x^l}F'\left(\mu_0;x_0\right),\quad n\to\infty,
$$
that contradicts~\eqref{f_P-F}.

The uniform convergence~\eqref{f_conv_of_P_n''} can be proved by the same argument taking into an account that 
\begin{align*}
&\frac{\partial^{|l|}}{\partial x^l} \frac{\partial^{|\tilde{l}|}}{\partial y^{\tilde{l}}}P_n''(F)(\mu;x,y)\\
&=\frac{\partial^{|l|}}{\partial x^l}\frac{\partial^{|\tilde{l}|}}{\partial y^{\tilde{l}}}\sum_{j_1,\ldots,j_d=0}^n\sum_{i_1,\ldots,i_d=0}^nF''\left(\chi_n(\mu_n);a_{j_1,\ldots,j_d}^n,a_{i_1,\ldots,i_d}^n\right)\varphi_{j_1,\ldots,j_d}^n(x)\varphi_{i_1,\ldots,i_d}^n(y)\\
&=\frac{\partial^{|l|}}{\partial x^l}\frac{\partial^{|\tilde{l}|}}{\partial y^{\tilde{l}}}\tilde{B}_n\left(F''\left(\chi_n(\mu_n);\cdot,\cdot\right)\right)(x,y),
\end{align*}
where $\tilde{B}_n$, $n\geq 1$, are the Bernstein polynomials defined for functions from $\Cf(K^2)$.
\end{proof}

\subsection{Approximation of differentiable functions on $\MF(\R^d)$}

We fix a smooth bounded function $\psi:\R^d\to\R$ and define a map from $\MF(\R^d)$ to $\MF(\R^d)$ as follows
$$
\vartheta_{\psi}(\mu)(dx):=\psi(x)\mu(dx).
$$ 

We also assume that $\psi$ has a compact support. Let $K=[a,b]^d$ such that $\supp\psi\subset K$. Then the measure $\vartheta_{\psi}(\mu)$ is supported on $K$ and, consequently, we can consider $\vartheta_{\psi}$ as a map from $\MF(\R^d)$ to $\MF(K)\subset\MF(\R^d)$.

\begin{lemma}\label{lem_cont_of_vartheta}
The map $\vartheta_{\psi}:\MF(\R^d)\to\MF(K)$ is continuous.
\end{lemma}

\begin{proof}
The proof trivially follows from the definition of $\vartheta_{\psi}$.
\end{proof} 

We define for each $F\in\Cf(\MF(K))$ a new function as follows
$$
\Gamma_{\psi}(F)(\mu):=\Gamma_{\psi,K}(F)(\mu):=F(\vartheta_{\psi}(\mu)),\quad \mu\in\MF(\R^d).
$$

\begin{lemma}\label{lemma_derivative_of_Gamma}
If $F\in\Cf^{k,m}(\MF(K))$ for some $k\in\{0,1,2\}$ and $m\geq 0$, then $\Gamma_{\psi}(F)\in\Cf^{k,m}(\MF(\R^d))$. Moreover,
$$
\Gamma_{\psi}'(F)(\mu;x)=\Gamma_{\psi}(F'(\cdot;x))(\mu)\psi(x)=F'(\vartheta_{\psi}(\mu);x)\psi(x),\quad \mu\in\MF(\R^d),\ \ x\in\R^d,
$$
and
$$
\Gamma_{\psi}''(F)(\mu;x)=\Gamma_{\psi}(F''(\cdot;x,y))(\mu)\psi(x)\psi(y)=F''(\vartheta_{\psi}(\mu);x,y)\psi(x)\psi(y),\quad \mu\in\MF(\R^d),\ \ x\in\R^d.
$$
\end{lemma}

\begin{remark}
We remark that $\psi(x)=0$ for all $x\in K^c:=\R^d\setminus K$, thus, we assume that the multiplication $f(x)\psi(x)=0$, even if $f$ is not defined for such $x$.
\end{remark}

\begin{proof}[Proof of Lemma~\ref{lemma_derivative_of_Gamma}]
The continuity of $\Gamma_{\psi}(F)$ immediately follows from Lemma~\ref{lem_cont_of_vartheta}. The derivatives of $\Gamma_{\psi}(F)$ can be computed using the following observation 
$$
\Gamma_{\psi}(F)(\mu+\eps\delta_x)=F(\vartheta_{\psi}(\mu)+\eps\psi(x)\delta_x),\quad \mu\in\MF(\R^d),\ \ x\in\R^d,\ \ \eps>0.
$$
\end{proof}

\begin{lemma}\label{lemma_conv_of_vartheta}
Let $\{\psi_n\}_{n\geq 1}$ be a sequence of uniformly bounded continuous functions on $\R^d$ which pointwise converges to $\psi\in\Cf_b(\R)$, then $\vartheta_{\psi_n}(\mu)\to\vartheta_{\psi}(\mu)$, $n\to\infty$, for each $\mu\in\MF$.
\end{lemma}

\begin{proof}
The lemma easily follows from the dominated convergence theorem.
\end{proof}

\begin{proposition}\label{proposition_convergence_Gamma}
Let $F\in\Cf^{k,m}(\MF(\R^d))$ for some $k\in\{0,1,2\}$ and $m\geq 0$. Let $\{\psi_n\}_{n\geq 1}$ be a sequence of smooth bounded functions on $\R^d$ such that $\psi_n\to\psi$ in $\Cf^m(\R^d)$, $n\to\infty$, and $\{\psi_n\}_{n\geq 1}$ is uniformly bounded. Then for each $\mu\in\MF(\R^d)$
$$
\Gamma_{\psi_n}(F)(\mu)\to\Gamma_{\psi}(F)(\mu),\quad n\to\infty,
$$
$$
\Gamma_{\psi_n}'(F)(\mu;\cdot)\to\Gamma_{\psi}'(F)(\mu;\cdot)\ \ \mbox{in}\ \ \Cf^m(\R^d),\quad n\to\infty,\quad\mbox{if}\ \ k\geq 1,
$$ 
and
$$
\Gamma_{\psi_n}''(F)(\mu;\cdot)\to\Gamma_{\psi}''(F)(\mu;\cdot)\ \ \mbox{in}\ \ \Cf^m(\R^{2d}),\quad n\to\infty,\quad\mbox{if}\ \ k=2.
$$ 
\end{proposition}

\begin{proof}
We first note that $F'(\mu_n;\cdot)\to F'(\mu;\cdot)$ in $\Cf^m(\R^d)$ and $F''(\mu_n;\cdot)\to F''(\mu;\cdot)$ in $\Cf^m(\R^{2d})$ as $\mu_n\to\mu$, if $F\in\Cf^{2,m}(\MF(\R^d))$. Thus, the statement immediately follows from lemmas~\ref{lemma_derivative_of_Gamma} and~\ref{lemma_conv_of_vartheta}.
\end{proof}

We denote by $\Cf_P^k(\MF(\R^d))$ the set of functions on $\MF(\R^d)$ of the form
$$
G(\mu)=u(\langle\varphi_1,\mu\rangle,\ldots,\langle\varphi_p,\mu\rangle),\quad \mu\in\MF(\R^d),
$$ 
where $\varphi_i$, $i\in[p]=\{1,\ldots,p\}$, are positive smooth functions with compact supports, ${u\in C^k([0,+\infty)^p)}$ and $p\in\N$. 

\begin{remark}\label{remark_about_extention}
We remark that a function belongs to $\Cf^k([0,+\infty)^p)$ if and only if it can be extended to a function from $\Cf^k(\R^p)$. 
\end{remark}

Let $F_K$ denote the restriction of a function $F$ from $\Cf(\MF(\R^d))$ to $\MF(K)$.

\begin{lemma}\label{lemma_restrict_of_F}
For each $F\in\Cf^{k,m}(\MF(\R^d))$ the function $F_K$ belongs to $\Cf^{k,m}(\MF(K))$ and
$$
F_K'(\mu;x)=F'(\mu;x),\quad \mu\in\MF(K),\ \ x\in K,
$$
$$
F_K''(\mu;x,y)=F'(\mu;x,y),\quad \mu\in\MF(K),\ \ x,y\in K.
$$
\end{lemma}

\begin{proof}
The proof of the lemma is trivial.
\end{proof}

\begin{theorem}\label{thm_approximation}
Let $F\in\Cf^{k,m}(\MF(\R^d))$ for some $k\in\{0,1,2\}$ and $m\geq 0$. Then there exists a sequence $\{F_n\}_{n\geq 1}$ from $\Cf_P^k(\MF(\R^d))$ such that for all $\mu\in\MF(\R^d)$
$$
F_n(\mu)\to F(\mu),\quad n\to\infty,
$$
$$
F_n'(\mu;\cdot)\to F'(\mu;\cdot)\ \ \mbox{in}\ \ \Cf^m(\R^d),\quad n\to\infty,\quad\mbox{if}\ \ k\geq 1,
$$ 
and
$$
F''_n(\mu;\cdot)\to F''(\mu;\cdot)\ \ \mbox{in}\ \ \Cf^m(\R^{2d}),\quad n\to\infty,\quad\mbox{if}\ \ k=2.
$$ 

Moreover, if for some $C>0$ the functions $F$, $F'$ and $F''$ and their derivatives are bounded on sets $\NM_C(\MF(\R^d))$, $\NM_C(\MF(\R^d))\times\R^d$ and  $\NM_C(\MF(\R^d))\times\R^{2d}$, respectively, then the sequence $\{F_n\}_{n\geq 1}$ can be chosen with $\{F_n\}_{n\geq 1}$, $\{F_n'\}_{n\geq 1}$, $\{F_n''\}_{n\geq 1}$ and their derivatives \ uniformly \ bounded \ in $n$ \ on \ $\NM_C(\MF(\R^d))$, \ $\NM_C(\MF(\R^d))\times\R^d$ \ and \ ${\NM_C(\MF(\R^d))\times\R^{2d}}$, respectively. 
\end{theorem}

\begin{proof}
We assume that $k=2$. Let $\psi_n$ be a sequence of smooth functions on $\R^d$ such that they take values from $[0,1]$, $\supp\psi_n\subset K_n:=[-n,n]^d$, $\psi_n(x)=1$, $x\in[-n+1,n-1]^d$, and all derivatives are uniformly bounded in $x$ and $n$, i.e. for each $l\in\N_0^d$, the set $\left\{\frac{\partial^{|l|}}{\partial x^l}\psi_n(x),\ \ x\in\R^d,\ \ n\geq 1\right\}$ is bounded. Let us fix a function $F\in\Cf^{k,m}(\MF(\R^d))$. We are going to approximate $F_{K_n}$ by polynomials introduced in the previous section. So, by Proposition~\ref{prop_convergence_in_C_M} and Theorem~\ref{theorem_approximation_by_bernstein_polynomials}, for every $n\geq 1$ there exists a number $N_n\in\N$ such that
$$
\sup_{\mu\in\NM_n({K_n})}\left|F_{K_n}(\mu)-P_{N_n}(F_{K_n})(\mu)\right|\leq\frac{1}{n},
$$
$$
\sup_{\mu\in\NM_n({K_n})}\left\|F_{K_n}'(\mu;\cdot)-P_{N_n}'(F_{K_n})(\mu;\cdot)\right\|_{\Cf^m(K_n)}\leq\frac{1}{n}
$$
and
$$
\sup_{\mu\in\NM_n({K_n})}\left\|F_{K_n}''(\mu;\cdot)-P_{N_n}''(F_{K_n})(\mu;\cdot)\right\|_{\Cf^m(K_n^2)}\leq\frac{1}{n},
$$
where $\NM_n({K_n})$ is defined in Lemma~\ref{lemma_NM_C} with $C=n$ and $K=K_n$, and $P_{N_n}$ is defined by~\eqref{f_P_n} for $K=K_n$.

We set $F_n(\mu):=\Gamma_{\psi_n}\left(P_{N_n}(F_{K_n})\right)(\mu)=P_{N_n}(F_{K_n})\left(\vartheta_{\psi_n}(\mu)\right)$, $\mu\in\MF(\R^d)$. By Lemma~\ref{lemma_derivative_of_Gamma}, $F_n\in\Cf^{k,m}(\MF(\R^d))$. Moreover, it is easy to see that $F_n\in\Cf_P^k(\MF(\R^d))$, by the definition of $P_{N_n}$ and $\Gamma_{\psi_n}$.

Next, we are going to show that $\{F_n\}_{n\geq 1}$ is the sequence which approximates $F$. We fix $\eps>0$, $\mu\in\MF(\R^d)$ and a compact set $K\subset\R^d$. We choose $\tilde{n}\in\N$ such that $\frac{1}{\tilde{n}}<\frac{\eps}{2}$, $K\subset K_{\tilde{n}}$, $\mu(\R^d)\leq\tilde{n}$ and for all $n\geq \tilde{n}$
$$
\left|F(\mu)-\Gamma_{\psi_n}(F)(\mu)\right|<\frac{\eps}{2},
$$
$$
\left\|F'(\mu;\cdot)-\Gamma_{\psi_n}'(F)(\mu;\cdot)\right\|_{C^m(K)}<\frac{1}{\eps}
$$
and
$$
\left\|F''(\mu;\cdot)-\Gamma_{\psi_n}''(F)(\mu;\cdot)\right\|_{C^m(K^2)}<\frac{1}{\eps}.
$$
Such $\tilde{n}$ exists due to Proposition~\ref{proposition_convergence_Gamma}, since $\{\psi_n\}_{n\geq 1}$ converges to the function $\psi=1$ in $\Cf^m(\R^d)$. Let us remark that $\Gamma_{\psi_n}(F)=\Gamma_{\psi_n}(F_{K_n})$, $\Gamma_{\psi_n}'(F)=\Gamma_{\psi_n}'(F_K)$ and $\Gamma_{\psi_n}''(F)=\Gamma_{\psi_n}''(F_K)$, by Lemma~\ref{lemma_restrict_of_F}.
So, now we can estimate for each $n\geq\tilde{n}$
\begin{align*}
\left|F(\mu)-F_n(\mu)\right|&\leq\left|F(\mu)-\Gamma_{\psi_n}(F)(\mu)\right|+\left|\Gamma_{\psi_n}(F_{K_n})(\mu)-\Gamma_{\psi_n}\left(P_{N_n}(F_{K_n})\right)(\mu)\right|\\
&\leq\frac{\eps}{2}+\left|F_{K_n}(\vartheta_{\psi_n}(\mu))-P_{N_n}(F_{K_n})(\vartheta_{\psi_n}(\mu))\right|\leq\frac{\eps}{2}+\frac{1}{n}\leq \eps,
\end{align*}
since $\vartheta_{\psi_n}(\mu)\in\NM_n({K_n})$. Similarly, for each $n\geq\tilde{n}+1$, we have
\begin{align*}
\left\|F'(\mu;\cdot)-F_n'(\mu;\cdot)\right\|_{C^m(K)}&\leq\left\|F'(\mu;\cdot)-\Gamma_{\psi_n}'(F)(\mu;\cdot)\right\|_{C^m(K)}\\
&+\left\|\Gamma_{\psi_n}'(F_{K_n})(\mu;\cdot)-\Gamma_{\psi_n}'\left(P_{N_n}(F_{K_n})\right)(\mu;\cdot)\right\|_{C^m(K)}\\
&\leq\frac{\eps}{2}+\left|F_{K_n}'(\vartheta_{\psi_n}(\mu);\cdot)\psi_n-P_{N_n}'(F_{K_n})(\vartheta_{\psi_n}(\mu;\cdot))\psi_n\right\|_{C^m(K)}\\
&\leq\frac{\eps}{2}+\frac{1}{n}\leq \eps,
\end{align*}
since $\psi_n(x)=1$ on $K$ for all $n\geq \tilde{n}+1$.
Analogously, $\left\|F''(\mu;\cdot)-F_n''(\mu;\cdot)\right\|_{C^m(K^2)}<\eps$ for all $n\geq\tilde{n}+1$. The theorem is proved.
\end{proof}

\providecommand{\bysame}{\leavevmode\hbox to3em{\hrulefill}\thinspace}
\providecommand{\MR}{\relax\ifhmode\unskip\space\fi MR }
\providecommand{\MRhref}[2]{%
  \href{http://www.ams.org/mathscinet-getitem?mr=#1}{#2}
}
\providecommand{\href}[2]{#2}




%
%

\end{document}